\documentclass{article}
\usepackage{preamble}
\usepackage{graphicx}
\usepackage[margin=3cm]{geometry}

\title{Hoffman colorability of (strongly) regular graphs}
\author{Aida Abiad\thanks{\texttt{a.abiad.monge@tue.nl},  Department of Mathematics and Computer Science, Eindhoven University of Technology, The Netherlands}\thanks{Department of Mathematics and Data Science of Vrije Universiteit Brussel, Belgium}\qquad Bart De Bruyn\thanks{\texttt{Bart.DeBruyn@ugent.be}, Department of Mathematics, Computer Science and Statistics, Ghent University, Belgium}  \qquad Thijs van Veluw\thanks{\texttt{Thijs.vanVeluw@ugent.be}, Department of Mathematics, Computer Science and Statistics, Ghent University, Belgium}\thanks{Department of Mathematics and Computer Science, Eindhoven University of Technology, The Netherlands}}

\date{}

\begin{document}

\maketitle

\begin{abstract}
Hoffman's bound is a well-known eigenvalue bound on the chromatic number of a graph. By interpreting this bound as a parameter, we show multiple applications of colorings attaining the bound (Hoffman colorings) for several notions of graph regularity: regular, (co-)edge-regular, and strongly regular. For strongly regular graphs, we prove that Hoffman colorability implies pseudo-geometricity, and we strengthen Haemers' finiteness result on strongly regular graphs with a bounded chromatic number by considering the Hoffman bound instead of the chromatic number. Furthermore, by using Hoffman colorings we show that a sufficient condition for non-unique vector colorability shown by Godsil, Roberson, Rooney, \v{S}\'amal and Varvitsiotis  [European J. Combin. 79, 2019] can be relaxed in the setting of strongly regular graphs. Lastly, using Hoffman colorings we derive several new characterizations of the mentioned graph regularity notions. \\

\end{abstract}

\section{Introduction}

Let $G$ be a simple graph. Write $\y_{\max}$, $\y_{\min}$ for the largest and the smallest eigenvalue respectively of the adjacency matrix of $G$. Then the \emph{Hoffman bound} \cite{Hoffman} states that the chromatic number of $G$ is bounded below by
\begin{equation}\label{eq:Hoffmanbound}
h(G) \coloneqq 1-\frac{\y_{\max}(G)}{\y_{\min}(G)},
\end{equation}
which we will refer to as the \emph{Hoffman number}. A graph for which the Hoffman bound is tight is called \emph{Hoffman colorable}, and every optimal coloring of such a graph is called a \emph{Hoffman coloring}. Bipartite graphs and regular complete multipartite graphs are easily seen to be Hoffman colorable, and we will refer to their optimal colorings as \emph{trivial Hoffman colorings}.

In \cite{Abiad,previouspaper}, Hoffman colorings of general graphs (also including irregular graphs) were studied, and several structural requirements on Hoffman colorings were obtained. For strongly regular graphs, Hoffman colorings are dual to the notion of a \emph{spread}, which was investigated in \cite{spreads}, where spreads (and in turn Hoffman colorings) of strongly regular graphs are shown to correspond to fans in partial geometries, imprimitive 3-class association schemes, and systems of linked symmetric designs. In \cite{Beers, CHP}, equality in an analogous Hoffman-type bound but for the normalized Laplacian matrix was studied.

In this paper, we further investigate Hoffman colorability for several graph regularity notions. In particular, we focus on regular, co-edge-regular and strongly regular graphs.

First, we show that Hoffman colorable strongly regular graphs and their complements are pseudo-geometric (\thref{thm:HCpseudoG}). Moreover, we show that there are only finitely many primitive strongly regular graphs with a bounded Hoffman number (\thref{thm:boundedHoffman}), strengthening the known result \cite{ProefschriftHaemers} that there are only finitely many primitive strongly regular graphs with a bounded chromatic number. We make this explicit for primitive strongly regular graphs with Hoffman number at most 3 by completely determining all six such graphs (\thref{cor:h<3}).

Next we investigate the interplay between unique vector colorability (which was introduced in \cite{PakVilenchik} and studied with respect to cores in \cite{ChiVector}) and Hoffman colorability. In particular, we show that non-trivially Hoffman colorable connected 2-walk-regular graphs (including primitive strongly regular graphs) are not uniquely vector colorable (\thref{thm:UVCnotHC}). This improves \cite[Theorem 3.12]{ChiVector} for strongly regular graphs by providing a more general sufficient condition for non-unique vector colorability that also includes the graphs of type A of the type system of \cite{SRGcores}, like the Shrikhande and Schläfli graphs.

Furthermore, with the Hoffman number in mind, we expand on the results and ideas of \cite{Neumaier}, where average parameters based on strongly regular graphs are used to introduce an upper bound on the least eigenvalue of a regular graph \cite[Theorem 2.14(a)]{Neumaier}. In particular, we show that for regular graphs it holds that
\begin{equation}\label{eq:hhc<n}
    h(G)h(\ol G)\le n,
\end{equation}
and that strongly regular graphs are precisely those regular graphs attaining equality (\thref{thm:characterization}). 
The inequality (2) is related to the Lov\'asz  $\vartheta$-number (Remark 5.5). We also obtain a lower bound on the chromatic number of a co-edge-regular graph that improves Hoffman's bound for this graph class (\thref{thm:co-edge}). As an application of these results to Hoffman colorable regular graphs, it follows that graphs that are co-edge-regular but not strongly regular cannot be Hoffman colorable, and that strictly Neumaier graphs are not Hoffman colorable (\thref{cor:sufficient}).

The rest of the paper is organized as follows. In Section \ref{sec:preliminaries} we recall some definitions and   results that we will use throughout the paper. In Section \ref{sec:geompara}, we introduce the geometric parameters for strongly regular graphs and derive the aforementioned results. In Section \ref{sec:UVC} we investigate the interplay between Hoffman colorability and unique vector colorability. Lastly, in Section \ref{sec:characterization}, we cover the average parameters from \cite{Neumaier} and show the aforementioned results.

\section{Preliminaries}\label{sec:preliminaries}
In this section we establish the notation, definitions and introduce some preliminary results that will be used in our proofs. Throughout, we write $n$ for the order of a graph and, if it is regular, $k$ for the valency. For the vertex set of a graph $G$ we write $V(G)$, and similarly for the edge set $E(G)$. We write $A$ for the adjacency matrix.

\subsubsection*{(Co-)edge-regular and strongly regular graphs}

A regular graph is \emph{$a$-edge-regular} if every edge is contained in a constant number $a$ of triangles. Equivalently, a regular graph is edge-regular if every pair of adjacent vertices shares a constant number $a$ of common neighbors. A graph is \emph{$c$-co-edge-regular} if every pair of non-adjacent vertices shares a common number $c$ of common neighbors. The complement of an $a$-edge-regular graph is $(n-2k+a)$-co-edge-regular, and the complement of a $c$-co-edge-regular graph is $(n-2k+c-2)$-edge-regular.

A $k$-regular graph $G$ on $n$ vertices is \emph{$(n,k,a,c)$-strongly regular} if it is $a$-edge-regular and $c$-co-edge-regular. The complement $\ol G$ of such a strongly regular graph $G$ is $(n,n-k-1,n-2k+c-2,n-2k+a)$-strongly regular. A strongly regular graph is \emph{primitive} if both $G$ and $\ol G$ are connected, and \emph{imprimitive} otherwise. The only imprimitive strongly regular graphs are disjoint unions of equal sized complete graphs, and regular complete multipartite graphs. By a double counting argument, it holds that
\begin{equation}\label{eq:srg}
    c(n-k-1)=k(k-a-1).
\end{equation}

Connected regular graphs are strongly regular if and only if they have at most three distinct eigenvalues. The largest eigenvalue is simple and equal to the valency $k$, and, for non-complete strongly regular graphs, we write $\tau < \theta$ for the other two eigenvalues. In this case, we have
\begin{equation}\label{eq:srgeigenvalues}
    (x-\theta)(x-\tau)=x^2+(c-a)x+(c-k).
\end{equation}
For primitive strongly regular graphs we have $\tau<-1$ and $\theta>0$. If $G$ has eigenvalues $\tau<\theta<k$, then $\ol G$ has eigenvalues $-1-\theta<-1-\tau<n-k-1$.

A \emph{conference graph} is a strongly regular graph that has the same parameters (and hence same spectrum) as its complement. If a strongly graph is not integral (in other words, if $\tau$ and $\theta$ are not integral), then it necessarily is a conference graph.

For more details about strongly regular graphs, see \cite{spectra} or \cite{topics8}.

\subsubsection*{(Co)clique and chromatic number bounds: Delsarte/ratio bound and Hoffman bound}

A \emph{clique} is a subset $C$ of the vertices of a graph such that every two distinct vertices in $C$ are adjacent. A \emph{coclique} (also called \emph{stable set} or \emph{independent set}) is a subset $C$ of the vertices such that no two vertices in $C$ are adjacent. The \emph{clique number} $\omega(G)$ and \emph{independence number} $\alpha(G)$ of a graph $G$ are the sizes of the largest clique and coclique, respectively. We call a (co)clique $C$ \emph{$e$-regular} if every vertex that is not in $C$ has a constant number $e$ of neighbors in $C$.

If $C$ is a coclique in a $k$-regular graph $G$ on $n$ vertices, then the \emph{ratio bound} (by Hoffman, unpublished, see \cite[Theorem 3.5.2]{spectra} and \cite{ratiobound}) states
\begin{equation}\label{eq:ratiobound}
    |C| \le \frac{n\y_{\min}(G)}{\y_{\min}(G)-k},
\end{equation}
with equality if and only if $C$ is $(-\y_{\min}(G))$-regular. In the case of strongly regular graphs, this was already known \cite{Delsarte} so in the context of strongly regular graphs the bound is called the \emph{Delsarte bound}. A coclique meeting the ratio (Delsarte) bound is called a \emph{Hoffman (Delsarte) coclique}. A clique that is a Hoffman (Delsarte) coclique in the complement graph is called a \emph{Hoffman (Delsarte) clique}.

Since a coloring of a graph is a partition of the vertices into $\chi$ cocliques, we have
\begin{equation}\label{eq:classic}
    \chi(G) \ge \frac{n}{\a(G)}.
\end{equation}
For regular graphs, the Hoffman bound (\ref{eq:Hoffmanbound}) can be deduced using (\ref{eq:classic}) and the ratio bound (\ref{eq:ratiobound}). This implies the following result.

\begin{prop}[{\cite[Proposition 2.3]{3chromDRG}}]\thlabel{prop:constantequitable}

If $G$ is regular, then every color class of a Hoffman coloring is a Hoffman coclique. In other words, the color classes have equal size and every vertex is adjacent to exactly $-\y_{\min}(G)$ vertices of every color other than its own.
\end{prop}

A \emph{spread} in a strongly regular graph is a partition of the vertices into Delsarte cliques. So for strongly regular graphs, \thref{prop:constantequitable} implies that a Hoffman coloring is equivalent to a spread in the complement, and this was the starting point of \cite{spreads}. Moreover, for strongly regular graphs a straightforward computation gives
\begin{equation}\label{eq:hhc=n}
    h(G)h(\ol G)=n.
\end{equation}
In \thref{lem:(s+1)(ols+1)=n} we will show a derivation of (\ref{eq:hhc=n}) in a more general setting. Now the Delsarte bound says that cliques have size at most the Hoffman number. In other words, for strongly regular graphs we have
\begin{equation}\label{eq:w<h<chi}
\omega(G)\le h(G)\le \chi(G).
\end{equation}

\subsubsection*{(Strict) vector colorings and the Lov\'{a}sz $\vartheta$-number}

Vector colorings were first introduced in \cite{KargerEtAl}. Given a graph $G$ and a real number $t\ge 2$, a \emph{vector $t$-coloring} is an assignment of unit vectors from a finite-dimensional real vector space to the vertices of $G$, such that for any pair of adjacent vertices the assigned vectors have inner product bounded above by $-1/(t-1)$. A vector coloring is \emph{strict} if for any pair of adjacent vertices, the inner product of the assigned vectors is equal to $-1/(t-1)$. Any graph coloring in the standard sense with $c$ colors induces a strict vector $c$-coloring through a regular simplex on $c$ vertices (see \cite[Lemma 4.1]{KargerEtAl}). The \emph{vector chromatic number} $\chi_v$ is defined as the least $t$ for which the graph admits a vector $t$-coloring. Similarly, the \emph{strict vector chromatic number} $\chi_{sv}$ is defined as the least $t$ for which the graph admits a strict vector $t$-coloring. We now have

\begin{equation}\label{eq:chiv<chisv<chi}
\chi_v(G)\le \chi_{sv}(G)\le \chi(G).
\end{equation}

In \cite{Lovasz}, a graph invariant $\vartheta(G)$ was introduced, later named the \emph{Lovász $\vartheta$-number}, as an upper bound for the Shannon capacity. In \cite{MERR,Schrijver} a variation $\vartheta'(G)$ was introduced, which has later been called the \emph{Schrijver variant of the Lov\'{a}sz $\vartheta$-number}. Vector colorings are known to be related to the Lovász $\vartheta$-number and its Schrijver variant. In particular, \cite[Theorem 8.2]{KargerEtAl} says that $\vartheta(\ol G)=\chi_{sv}(G)$ for all graphs, and furthermore in \cite{ChiVector,SabidussivsHedetniemi} is stated that $\vartheta'(\ol G)=\chi_v(G)$. In the following, we will switch between the vector chromatic notation $\chi_v,\chi_{sv}$ and Lovász/Schrijver notation $\vartheta,\vartheta'$, according to the context.

From \cite[Theorem 6]{Lovasz} it follows that $h(G)\le \chi_{sv}(G)$. In fact, the same argument works for the Schrijver variant of the Lovász $\vartheta$-number, so we get $h(G) \le \chi_v(G)$. Together with (\ref{eq:chiv<chisv<chi}), we obtain that the (strict) vector chromatic number is sandwiched between the Hoffman number and the standard chromatic number:
\begin{equation}\label{eq:sandwich}
    h(G) \le \chi_v(G) \le \chi_{sv}(G) \le \chi(G).
\end{equation}

\section{Strongly regular graphs and geometric parameters}\label{sec:geompara}

In this section we show two main results: that Hoffman colorable strongly regular graphs and their complements are pseudo-geometric (\thref{thm:HCpseudoG}), and that there are only finitely many primitive strongly regular graphs with a bounded Hoffman number (\thref{thm:boundedHoffman}). Compared to the known result \cite{ProefschriftHaemers} that there are only finitely many primitive strongly regular graphs with a given chromatic number, \thref{thm:boundedHoffman} shows that this conclusion holds using a more relaxed condition. We also determine all primitive strongly regular graphs with Hoffman number at most 3 (\thref{cor:h<3}). To do this we study strongly regular graphs using parameters that are based on $(s,t,\a)$-partial geometries. In \cite{Neumaier, NeumaierBound}, steps were already made in this direction, see \thref{rmk:Neumaier}. These new ``geometric parameters'' $(s,t,\a)$ are completely interchangeable with the standard ``combinatorial parameters'' $(n,k,a,c)$ and the ``spectral parameters'' $(k,\theta,\tau)$. The relevance of the geometric parameters, and the connection to Hoffman colorability, lies in the fact that the Hoffman number is very easily expressible in the geometric parameters; it is $s+1$ (see (\ref{eq:h=s+1})).

We first briefly recall some definitions regarding partial geometries and (pseudo-)geometricity of strongly regular graphs. An \emph{incidence structure} is a triple $(P,L,I)$ with a set of \emph{points} $P$, a set of \emph{lines} $L$, and $I\subseteq P \times L$ a relation. If $(p,\ell)\in I$, then we write $pI\ell$ and say that $p$ and $\ell$ are \emph{incident}. Two points are \emph{collinear} if they are incident to a common line. The \emph{collinearity graph} is the graph with the points as vertices, such that adjacency is given by collinearity.

Let $s, t,  \a$ be positive integers. A \emph{partial geometry} with parameters $(s,t,\a)$ is an incidence structure $(P,L,I)$, with the following requirements:
\begin{itemize}
    \item Every pair of points is incident to at most one common line.
    \item Every line is incident to exactly $s+1$ points, and every point is incident to exactly $t+1$ lines.
    \item If $p$ is not incident to $\ell$, then $p$ is collinear to exactly $\a$ points incident to $\ell$.
\end{itemize}

It is known  \cite[Theorem 4.1]{topics8} that the collinearity graph of a partial geometry is strongly regular with parameters equal to
\begin{equation}\label{eq:CPPG}
    (n,k,a,c) = \Big ((s+1)\frac{st+\a}\a,s(t+1),s-1+t(\a-1),\a(t+1)\Big ).
\end{equation}
 A graph that is the collinearity graph of a partial geometry is called \emph{geometric}. A strongly regular graph for which there exist positive integers $s,t,\a$ such that the parameters are given by (\ref{eq:CPPG}) is called \emph{pseudo-geometric}, see \cite{spectra} or \cite{topics8}. The eigenvalues of a pseudo-geometric strongly regular graph are given by $k=s(t+1)$, $\theta=s-\a$, $\tau=-t-1$. Equivalently, we have
\begin{equation}\label{eq:geometric parameters}
    (s,t,\a) = \Big (-\frac k \tau, -\tau-1, -\frac k\tau - \theta \Big ).
\end{equation}

In order to show our results, we extend (\ref{eq:geometric parameters}) to all primitive strongly regular graphs, including non-pseudo-geometric graphs: we define the \emph{geometric parameters} $(s,t,\a)$ of a general primitive strongly regular with eigenvalues $k>\theta>\tau$ as in (\ref{eq:geometric parameters}). From (\ref{eq:srg}) and (\ref{eq:srgeigenvalues}), it follows that the combinatorial parameters of a primitive strongly regular graph with geometric parameters $(s,t,\a)$ are given by (\ref{eq:CPPG}). This implies that the geometric parametrization is as complete as the combinatorial and spectral parametrizations. The advantage of the geometric parametrization over the other two parametrizations is the ease of expressing the Hoffman number:
\begin{equation}\label{eq:h=s+1}
    h=s+1.
\end{equation}

Furthermore, one could say that the fact that (\ref{eq:CPPG}) holds for every primitive strongly regular graph implies that every strongly regular graph is pseudo-geometric in a way, if we relax the requirement that $s$, $t$, and $\a$ in (\ref{eq:CPPG}) are positive integers. However, we reserve the term ``pseudo-geometric'' for the case where the geometric parameters are positive integers. Note that \cite{spectra} and \cite{topics8} define geometric and pseudo-geometric also for the imprimitive strongly regular graphs, but we do not consider this case.

\begin{rmk}\thlabel{rmk:Neumaier}
The geometric parameters correspond to $K-1$, $m-1$, $e$ respectively, which were introduced in \cite{NeumaierBound}. However, this was only used in the context of edge-regular graphs with a regular clique (\emph{Neumaier graphs}, also see Section \ref{sec:characterization}). Later, in \cite{Neumaier}, $K$ reappears as $s+1$, now in the context of general edge-regular graphs (with or without a regular clique). However, the connection to partial geometries and the Hoffman number of strongly regular graphs was not noted. In Section \ref{sec:characterization}, we will use $s$ again, in the context of regular graphs.    
\end{rmk}

If a strongly regular graph $G$ is geometric, then the geometric parameters of $G$ are equal to the parameters of the partial geometry of which the collinearity graph is $G$. Consequently, the parameters of the square lattice graph $L(K_{m,m})$ are $(m-1,1,1)$, and the parameters of the triangular graph $L(K_m)$ are $(m-2,1,2)$. For general primitive strongly regular graphs, the geometric parameters $s$, $t$, and $\a$ are all positive, since $\tau<-1$ and $c>0$ for primitive strongly regular graphs. However, they might not be integral, or even rational. We can explain this trichotomy in terms of the rationality/integrality of the Hoffman number $h$ of the primitive strongly regular graph.
\begin{prop}\thlabel{prop:trichotomy}
    Let $G$ be a primitive strongly regular graph with eigenvalues $k>\theta>\tau$ and geometric parameters $(s,t,\a)$. Then the following are equivalent:
    \begin{enumerate}[label=(\roman*)]
        \item the Hoffman number of $G$ is rational,
        \item $s$ and $\a$ are rational and $t$ is integral,
        \item $G$ is integral (i.e. $\theta$ and $\tau$ are integers).
    \end{enumerate}
    Moreover, the Hoffman number of $G$ is an integer if and only if $G$ is pseudo-geometric.
\end{prop}
\begin{proof}
For the first set of equivalences, note that (iii) $\Rightarrow$ (ii) $\Rightarrow$ (i) follows from (\ref{eq:geometric parameters}) and (\ref{eq:h=s+1}). For (i) $\Rightarrow$ (iii), note that $\tau=-k/(h(G)-1)$ is rational, and since $\tau$ is an algebraic integer, this implies that $\tau$ is an integer, and hence also $\theta$.

For the final equivalence, note that by (\ref{eq:h=s+1}) it follows that a pseudo-geometric graph has an integral Hoffman bound. For the other direction, suppose that $h(G)$ is integral. Then by (\ref{eq:h=s+1}) $s$ is integral. By the first set of equivalences also $t$ and $\a=s-\theta$ are integral, so that $G$ is pseudo-geometric.
\end{proof}

By \thref{prop:trichotomy}, the strongly regular graphs with an irrational Hoffman number are precisely the non-integral strongly regular graphs (which are necessarily conference graphs), for instance the pentagon. An example of a strongly regular graph with a rational but not integral Hoffman number is the Petersen graph.

Using \thref{prop:trichotomy} we can show that Hoffman colorable strongly regular graphs and their complements are pseudo-geometric. We also have a similar but weaker statement for strongly regular graphs with a Delsarte clique.
\begin{thm}\thlabel{thm:HCpseudoG}
    Let $G$ be a primitive strongly regular graph. Then $G$ is pseudo-geometric if $G$ is Hoffman colorable or has a Delsarte clique. In the first case, also $\ol G$ is pseudo-geometric.
\end{thm}
\begin{proof}
    Let $G$ be a primitive strongly regular graph. If $G$ is Hoffman colorable or has a Delsarte clique, then one of the two inequalities of (\ref{eq:w<h<chi}) is an equality. Hence the Hoffman number is integral, and by \thref{prop:trichotomy} $G$ is pseudo-geometric. In particular, if $G$ has a spread, then it has a Delsarte clique, so it is pseudo-geometric. The last part of the result now follows from the fact that a Hoffman coloring in a strongly regular graph is equivalent to a spread in the complement \cite{spreads}.
\end{proof}

If $G$ is a primitive strongly regular graph with geometric parameters $(s,t,\a)$, then the geometric parameters $(\ol s, \ol t, \ol \a)$ are given by
\begin{equation}\label{eq:GPcomp}
    (\ol s, \ol t, \ol \a) = \Big ( \frac {st}\a , s-\a, \frac{t(s-\a)}{\a} \Big ),
\end{equation}
and hence the complement of a pseudo-geometric graph is pseudo-geometric again (as is the case in \thref{thm:HCpseudoG}) if and only if $\a$ divides $st$.

\begin{rmk}\thlabel{rmk:alpha}
Let $G$ be a primitive strongly regular graph with a Delsarte clique $C$ (so that $s$, $t$, and $\a$ are integers by \thref{thm:HCpseudoG}). We know that $C$ is of size $s+1$. In the complement $\ol G$, $C$ is a Delsarte coclique, so every vertex outside $C$ is adjacent to $-\y_{\min}(\ol G)=\theta+1$ vertices of $C$ (see (\ref{eq:ratiobound})). This implies that in $G$ every vertex outside of $C$ is adjacent to precisely $(s+1)-(\theta+1)=\a$ vertices of $C$. Note that this combinatorial interpretation of $\a$ agrees with its definition in the context of partial geometries.
\end{rmk}

\begin{exa}
An example of \thref{thm:HCpseudoG} where both $G$ and $\ol G$ are geometric is $L(K_6)$, of which the chromatic number and Hoffman number both equal 5. $\ol {L(K_6)}$ is the collinearity graph of a $(2,2,1)$-partial geometry. Note however that only $L(K_6)$ is Hoffman colorable.
\end{exa}

It is known that there are only finitely many primitive strongly regular graphs with a given chromatic number \cite[Theorem 4.1.2]{ProefschriftHaemers}. We strengthen this result by showing that there are finitely many primitive strongly regular graphs with a bounded Hoffman number. In order to do this, we need the geometric parameters and the following well-known result, which, for convenience, we rephrase in the context of this section.

\begin{thm}[{\cite[Theorem 9.1.9]{spectra}}]\thlabel{thm:boundedtau}
    Let $m\in \N$. Then, apart from geometric graphs with $\a\in \{t,t+1\}$, there are only finitely many primitive strongly regular graphs with smallest eigenvalue $\tau=-m$.
\end{thm}
\begin{thm}\thlabel{thm:boundedHoffman}
    Let $m$ be a positive number. Then there are only finitely many primitive strongly regular graphs with Hoffman number at most $m$.
\end{thm}
\begin{proof}
Let $G$ be a primitive strongly regular graph with $h(G)\le m$. We consider two cases.

If $G$ is a conference graph, then $h(G)=h(\ol G)$ and so by (\ref{eq:hhc=n}) $h(G)=\sqrt {n}$, so that the number of vertices of such a graph is bounded.

If $G$ is not a conference graph, then it is integral by \cite[Theorem 9.1.3(iv)]{spectra}. Using (\ref{eq:geometric parameters}) and (\ref{eq:h=s+1}), we have $\theta= s-\a <s =h(G)-1\le m-1$, and so $\ol \tau = -1-\theta >-m$. Applying \thref{thm:boundedtau} to every integer at most $m$ yields a finite set $\mathcal G$ such that either $\ol G \in \mathcal G$ or $\ol G$ is a geometric graph with $\ol \a\in \{\ol t,\ol t+1\}$. We claim that there are finitely many options in the second case as well. To show this, we prove that all three of $\ol s$, $\ol t$, and $\ol \a$ are bounded positive integers, implying that there are only finitely many possible parameter sets for such $\ol G$. Since $\ol G$ is geometric, all three of $\ol s$, $\ol t$, and $\ol \a$ are positive integers. Note that $\ol t = s-\a$ is bounded as above, and hence also $\ol \a \le \ol t +1$ is bounded. For $\ol s$, we have by (\ref{eq:GPcomp}) that $\ol s = s \ol \a/ \ol t \le 2s $ and so $\ol s$ is also bounded, concluding the proof.
\end{proof}

We make this explicit for the case where the Hoffman number is bounded above by 3.

\begin{cor}\thlabel{cor:h<3}
The primitive strongly regular graphs with Hoffman number at most 3 are the following (together with their geometric parameters):
    \begin{itemize}
        \item the pentagon: $\Big (\sqrt 5 -1, \frac{\sqrt 5 -1}2,\frac{\sqrt 5 -1}2 \Big )$,
        \item $L(K_{3,3})$: $(2,1,1)$,
        \item the Petersen graph: $\big (\frac 32, 1,\frac12 \big )$,
        \item $\ol {L(K_6)}$: $(2,2,1)$,
        \item the complement of the Clebsch graph: $\big (\frac 53,2,\frac 23\big )$,
        \item the complement of the Schläfli graph: $(2,4,1)$.
    \end{itemize}
\end{cor}
\begin{proof}
    We follow the proof of \thref{thm:boundedHoffman}. If $G$ is a conference graph, then it has at most nine vertices, giving only the pentagon and $L(K_{3,3})$ (see \cite[Section 9.9]{spectra}). If $G$ is not a conference graph, then it is integral, and we get $\ol \tau >-3$, so that $\ol \tau =-2$. The strongly regular graphs with least eigenvalue $-2$ are classified, see \cite[Theorem 9.2.1]{spectra}, from which the rest of the list follows.
\end{proof}

Note that the primitive strongly regular graphs with chromatic number 3 are known \cite{ProefschriftHaemers}; they are the pentagon, $L(K_{3,3})$, and the Petersen graph.

\thref{thm:boundedHoffman} also implies the following.

\begin{cor}
    For every natural number $c$, there are finitely many primitive strongly regular graphs that have a Delsarte clique of size $c$.
\end{cor}

By \thref{cor:h<3}, the primitive strongly regular graphs with Hoffman number 3 are $L(K_{3,3})$, $\ol{L(K_6)}$, and the complement of the Schläfli graph. Note that these three graphs are geometric, and hence all have a Delsarte clique. Thus, the primitive strongly graphs with a Delsarte clique of size 3 are precisely the three graphs with Hoffman number 3.

\section{Vector colorings}\label{sec:UVC}

Next we show a connection between Hoffman colorability and vector colorings that extends \cite[Theorem 3.12]{ChiVector} (see \thref{thm:UVCcore}) for the case of strongly regular graphs (\thref{thm:UVCnotHC}) to also cover the graphs of type A of the type system of \cite{SRGcores}, namely the Hoffman colorable strongly regular graphs without a Delsarte clique.

In order to state and prove \thref{thm:UVCnotHC}, we need the following definitions. A graph is \emph{uniquely vector colorable} if any two optimal vector colorings differ by an orthogonal transformation. A graph is \emph{$k$-walk-regular} if for every $m\in \N$ and every $0\le i \le k$, all pairs of vertices $u$ and $v$ at distance $i$ have the same number of walks of length $m$ from $u$ to $v$, so that $(A^m)_{uv}$ only depends on the distance $d(u,v)$, provided that this distance is at most $k$. Note that strongly regular graphs are $k$-walk-regular for any $k$. Lastly, a coloring (vector or standard) is \emph{locally injective} if vertices at distance 2 are assigned different colors. 

\begin{thm}\thlabel{thm:UVCnotHC}
    A non-trivially Hoffman colorable connected 2-walk-regular graph is not uniquely vector colorable.
\end{thm}
\begin{proof}
    Let $G$ be non-trivially Hoffman colorable and 2-walk-regular.
    
    Since $G$ is non-trivially Hoffman colorable, $G$ is neither a complete graph (Hoffman coloring is trivial) nor an odd cycle (among which only $K_3$ is Hoffman colorable). Therefore Brooks' Theorem $\chi(G)\le \Delta(G)$ implies that any Hoffman coloring $f$ of $G$ is not locally injective. As already mentioned in the Preliminaries, by \cite[Lemma 4.1]{KargerEtAl} $f$ induces a vector coloring $\widehat f$. Then $\widehat f$ is not locally injective (as $f$ is not), and by (\ref{eq:sandwich}) $\widehat f$ is optimal.

    Next, since $G$ is 1-walk-regular, \cite[Corollary 4.11]{canonical} gives an optimal vector coloring $\widetilde f$. Since $G$ is 2-walk-regular, $\widetilde f$ is locally injective by \cite[Lemma 3.11]{ChiVector}.

    We thus have two optimal vector colorings $\widehat f, \widetilde f$, of which $\widetilde f$ is locally injective and $\widehat f$ is not. We conclude that $G$ is not uniquely vector colorable.
\end{proof}

In \cite{ChiVector}, unique vector colorability of 2-walk-regular graphs was related to cores. A graph $G$ is a \emph{core} if every graph endomorphism (in other words, adjacency preserving mapping $f:V(G)\to V(G)$) is an automorphism (for more information on cores, see \cite{GodsilRoyle}). In particular, in \cite{ChiVector}, the following is proven.

\begin{thm}[{\cite[Theorem 3.12]{ChiVector}}]\thlabel{thm:UVCcore}
Let $G$ be a 2-walk-regular graph that is neither bipartite nor complete multipartite. If $G$ is uniquely vector colorable, then $G$ is a core.
\end{thm}

For strongly regular graphs, \thref{thm:UVCnotHC} is an extension of \thref{thm:UVCcore}. To see this, we need the following.

\begin{thm}[{\cite[Corollary 4.2]{SRGcores}}]\thlabel{thm:SRGcores}
    A non-complete strongly regular graph is not a core if and only if it is Hoffman colorable and has a Delsarte clique.
\end{thm}

\thref{thm:UVCnotHC} provides a new and more general sufficient condition for non-unique vector colorability than \thref{thm:UVCcore} in the case of strongly regular graphs. Indeed, if $G$ is not a core (the sufficient condition for non-unique vector colorability from \thref{thm:UVCcore}), then by \thref{thm:SRGcores} it is Hoffman colorable (the sufficient condition from \thref{thm:UVCnotHC}).

The question of which strongly regular graphs are cores but not uniquely vector colorable (so the non-uniquely vector colorable graphs for which non-unique vector colorability does not follow from \thref{thm:UVCcore}) was posed in \cite{ChiVector}. \thref{thm:UVCnotHC} provides examples for this. By \thref{thm:SRGcores}, the graphs for which non-unique vector colorability does follow from \thref{thm:UVCnotHC}, but does not follow from \thref{thm:UVCcore}, are precisely the strongly regular graphs of type A from the type system of \cite{SRGcores}: the Hoffman colorable strongly regular graphs that do not have a Delsarte clique. For instance, the following strongly regular graphs (with combinatorial parameters) are of type A:
\begin{itemize}
    \item the Shrikhande graph: $(16,6,2,2)$,
    \item the Schläfli graph: $(27,16,10,8)$,
    \item the three Chang graphs: $(28,12,6,4)$,
    \item the complements of six Latin square graphs: $(36,20,10,12)$.
\end{itemize}
In fact, for the parameter set $(36,20,10,12)$ the twelve non-uniquely vector colorable graphs (see \cite[Table 1]{ChiVector}) are precisely the complements of the twelve Latin square graphs of order 6. All of these are Hoffman colorable, but only six of them are cores. \thref{thm:UVCnotHC} therefore has a broader applicability than \thref{thm:UVCcore} for strongly regular graphs.

Following \cite[Section 3.5]{ChiVector} more closely, we see that we can actually prove in the same way that Taylor graphs are not Hoffman colorable. A \emph{Taylor graph} is a distance-regular graph with intersection array $\{k,\mu,1;1,\mu,k\}$ for some $k$ and $\mu$. A Taylor graph is of diameter 3, so never complete multipartite. A Taylor graph is bipartite if and only if $\mu=k-1$, if and only if we can obtain the graph by removing a perfect matching from a complete bipartite graph.
\begin{cor}\thlabel{cor:Taylor}
    If a Taylor graph is Hoffman colorable, then it is bipartite.
\end{cor}
\begin{proof}
    By \cite[Theorem 3.14]{ChiVector}, any Taylor graph is uniquely vector colorable. Since Taylor graphs are 2-walk-regular, any Hoffman coloring is trivial by \thref{thm:UVCnotHC}, implying that the graph is bipartite.
\end{proof}

\section{Characterizations for strongly regular graphs and a bound for the chromatic number of co-edge-regular graphs}\label{sec:characterization}

Here we investigate the Hoffman number in terms of the average parameters introduced in \cite{Neumaier}. Using this, we provide a new characterization of strongly regular graphs among regular graphs (\thref{thm:characterization}), namely these are precisely those regular graphs attaining equality in (\ref{eq:hhc<n}). Recall that (\ref{eq:hhc<n}) can also be shown using the Lov\'{a}sz $\vartheta$-number (see \thref{rmk:Lovasz}). Also, observe that (\ref{eq:hhc=n}) shows that for strongly regular graphs we have equality in (\ref{eq:hhc<n}); what we show in \thref{thm:characterization} is that strongly regular graphs are the only such regular graphs. Furthermore, we improve Hoffman's bound for the case of co-edge-regular graphs (\thref{thm:co-edge}). We conclude this section by discussing some general consequences of the new results for Hoffman colorable regular graphs. For instance, from \thref{thm:co-edge} it follows that Hoffman colorable co-edge-regular and Neumaier graphs must be strongly regular (\thref{cor:sufficient}(i) and (ii)).

\subsubsection*{A characterization of strong regularity using the Hoffman number}

In order to prove \thref{thm:characterization}, we use average parameters based on strongly regular graphs. Let $G$ be a $k$-regular graph on $n$ vertices, and assume for the remainder of this section that $G$ is neither empty nor complete. Let $\av a$ be the average number of common neighbors of two adjacent vertices of $G$, and similarly let $\av c$ be the average number of common neighbors of two distinct non-adjacent vertices. Let $\av \tau$ and $\av \theta$ be the smallest and largest root respectively of the polynomial
\begin{equation}\label{eq:avtheta}
g(x) = x^2+(\av c - \av a)x + (\av c - k),
\end{equation}
which is just (\ref{eq:srgeigenvalues}) with $\av a $ and $\av c$ instead of $a$ and $c$. Now let $\av s $ be $-k/\av \tau$, as in Section \ref{sec:geompara}.

The following straightforward result, which shows that these average parameters behave much in the same way as the parameters for strongly regular graphs, is later used for proving \thref{lem:(s+1)(ols+1)=n}.
\begin{lem}\thlabel{lem:average}
    Let $G$ be a $k$-regular graph, and let $\av a$, $\av c$, $\av \tau$, $\av \theta$, and $\av s$ be defined as above. Let $\av {\ol a}$, $\av{\ol c}$, $\av {\ol \tau}$, $\av {\ol \theta}$, and $\av {\ol s}$ those average parameters for the complement of $G$. Then the following hold.
    \begin{enumerate}[label=(\roman*)]
        \item $\av c(n-k-1)=k(k-\av a-1)$,
        \item $\av {\ol a}=n-2-2k+\av c$ and $\av {\ol c}=n-2k+\av a$,
        \item $\av{\ol \tau}=-1-\av \theta$ and $\av{\ol\theta}=-1-\av\tau$.
        \end{enumerate}
\end{lem}

\begin{proof}
    Part (i) follows from double counting the set $\{(u,v,w) \in V(G)^3: u \sim v, v \sim w, u \not \sim w, u\ne w\}$. Part (ii) follows by applying the same reasoning as for the corresponding statement for strongly regular graphs and averaging over every (non)-edge. For Part (iii), let $g(x)$ be the polynomial from (\ref{eq:avtheta}), and let $\ol g (x)$ be that for the complement of $G$ (which is $(n-k-1)$-regular). From Part (ii) it follows that $\ol g(x)=g(-x-1)$, from which Part (iii) follows.
\end{proof}

\begin{rmk}\thlabel{rmk:sameparameters}
Even though we define the average parameters in a different way than in \cite{Neumaier}, we can still apply the results from that paper: from \thref{lem:average} it follows that the average parameters defined here are equivalent to the parameters defined in \cite{Neumaier}; the parameters $\av a$, $\av c$, $\av \tau$, $\av \theta$, and $\av s$ correspond to $\ol \y$, $\ol \mu$, $\theta_m$, $\theta_M$, and $\ol s$ respectively.
\end{rmk}

Note that by either (\ref{eq:CPPG})-(\ref{eq:GPcomp}) or (\ref{eq:hhc=n})-(\ref{eq:h=s+1}) we have $(s+1)(\ol s+1)=n$ for strongly regular graphs. This remains true when using the average parameters for general regular graphs, as we now show.

\begin{lem}\thlabel{lem:(s+1)(ols+1)=n}
    Let $G$ be a regular graph, with $\av s$ defined as above. Then $(\av s+1)(\av{\ol s}+1)=n$.
\end{lem}
\begin{proof}
    Note that $g(k)=n\av c$ by \thref{lem:average}(i) with $g$ the polynomial from (\ref{eq:avtheta}). We compute
    \begin{align*}
        (\av s+1 ) ( \av {\ol s}+1) &= \Big (\frac{\av \tau - k}{\av \tau} \Big ) \Big ( \frac{n + \av \theta -k}{\av \theta+1} \Big )=\frac{n (\av \tau -k) + g(k)}{\av \tau ( \av \theta+1)}=\frac{n(\av \tau + \av c-k)}{\av \tau + \av \theta \av \tau}=n,
    \end{align*}
    using \thref{lem:average}(iii) and (\ref{eq:avtheta}).
\end{proof}

For irregular graphs, it is possible to show a statement similar to \thref{lem:average} regarding the average degree $\av k$, but with strict inequalities instead of equalities (with $>$ everywhere, except for the second equality of Part (ii) which has $<$). Also \thref{lem:(s+1)(ols+1)=n} holds for irregular graphs with $>$ instead of $=$ for irregular graphs.

We are now ready to state and prove the new characterization of strongly regular graphs using the Hoffman number.

\begin{thm}\thlabel{thm:characterization}
    Let $G$ be a regular graph. Then $h(G)h(\ol G)\le n$, with equality if and only if $G$ is strongly regular.
\end{thm}
\begin{proof}
By \thref{rmk:sameparameters}, we can apply the results of \cite{Neumaier}. In particular, by \cite[Theorem 2.14(a)]{Neumaier}, we have $\y_{\min}(G) \le \av \tau$ with equality if and only if $G$ is strongly regular. As noted in the proof of \cite[Theorem 2.30]{Neumaier}, by definition of $h(G)$ and $\av s$, this implies that
    \begin{equation}\label{eq:h<s+1}
        h(G) \le \av s+1,
    \end{equation}
    with equality if and only if $G$ is strongly regular. Next, by \thref{lem:(s+1)(ols+1)=n}, we have $(\av s+1)(\av {\ol s}+1)= n $. Applying (\ref{eq:h<s+1}) to both $G$ and $\ol G$ provides a proof for (\ref{eq:hhc<n}). If we have equality in (\ref{eq:hhc<n}), then by the above reasoning we must have equality in (\ref{eq:h<s+1}), from which it follows that $G$ is strongly regular.
\end{proof}

Alternative proofs of \thref{thm:characterization} are also possible. From \thref{lem:(s+1)(ols+1)=n} and (\ref{eq:h<s+1}) applied to $\ol G$ it follows that $\av s+1$ satisfies the ratio bound for cliques in $G$, see (\ref{eq:ratiobound}):
\begin{equation}\label{eq:s ratio}
    \av s+1 \le \frac n{h(\ol G)},
\end{equation}
with equality if and only if $G$ is strongly regular. This was also implicitly shown in the proof of \cite[Theorem 2.30]{Neumaier}. Now any choice of two from \thref{lem:(s+1)(ols+1)=n}, (\ref{eq:h<s+1}), and (\ref{eq:s ratio}) is sufficient to prove \thref{thm:characterization}.

We were recently pointed out by Roberson that \thref{thm:characterization} and the inequality \cite[Lemma 4.4]{abgraphs} can be shown to be equivalent, even though these results have completely different proofs.

\begin{rmk}\thlabel{rmk:Lovasz}
As mentioned before, the Lov\'{a}sz $\vartheta$-number is subject to inequalities that are related to the results discussed in this section. For regular graphs, the ratio bound (\ref{eq:ratiobound}) also holds for the Lovász $\vartheta$-number \cite[Theorem 9]{Lovasz}: if $G$ is regular, then
\begin{equation}\label{eq:Lovaszratio}
    \vartheta(G) \le \frac n {h(G)}.
\end{equation}
Considering a graph and its complement together, we have the following inequality \cite[Corollary 2]{Lovasz}:
\begin{equation}\label{eq:thetathetac>n}
    \vartheta(G)\vartheta(\ol G) \ge n.
\end{equation}
Combining (\ref{eq:Lovaszratio}) with either (\ref{eq:sandwich}) or (\ref{eq:thetathetac>n}), we get (\ref{eq:hhc<n}). Note that even though \cite{Lovasz} introduces sufficient conditions for equality in (\ref{eq:Lovaszratio}, \ref{eq:thetathetac>n}), necessary conditions are not given. Therefore, we cannot use the results of \cite{Lovasz} to characterize equality in (\ref{eq:hhc<n}) and prove \thref{thm:characterization}.
\end{rmk}

Note that we need the regularity requirement in \thref{thm:characterization}: take for example the star graph $G=K_{1,m}$ for $m\ge 2$. It is bipartite, so $h(G)=2$, and $\ol G$ is $K_m$ together with an isolated vertex, so $h(\ol G)=m$. The product of the Hoffman numbers is $2m$, which is greater than the number $m+1$ of vertices of $G$. The path graph on four vertices is bipartite and self-complementary, providing an irregular example for which $h(G)h(\ol G)=n$. The path graph on five vertices has $h(G)h(\ol G) \approx 4.48$, which is less than the number of vertices. In general, it seems hard to say anything about irregular graphs in this context.

\subsubsection*{A Hoffman-like bound for co-edge-regular graphs}

We can also use the above average parameters to give an improvement of the Hoffman bound for co-edge-regular graphs.

As mentioned in \thref{rmk:sameparameters}, the average parameters introduced in this section were already introduced in \cite{Neumaier}. Note however that the order in which the parameters are defined here is different from that in \cite{Neumaier}. In particular, $\av s$ can be defined using only $\av a$, as the positive root of
\begin{equation}\label{eq:s}
    (n+\av a -2k)x^2+(k^2-k+\av a -\av a n)x-k(n-k-1).
\end{equation}
For $a$-edge-regular graphs, we can therefore write $s$ instead of $\av s$. We have the following proposition, which was implicitly shown in \cite{NeumaierBound}, and explicitly in \cite{Neumaier}.
\begin{prop}[Neumaier bound, {\cite[Proposition 2.3]{Neumaier}}]\thlabel{prop:Neumaier}
    If $G$ is edge-regular, then
    $$\omega(G) \le s+1.$$
    A clique meets this bound if and only if it is regular.
\end{prop}
By (\ref{eq:s ratio}), this bound is a strict improvement of the ratio bound (\ref{eq:ratiobound}), except for strongly regular graphs for which the Neumaier bound and the ratio bound (\ref{eq:ratiobound}) are equivalent, in which case it was called the Delsarte bound.

A clique attaining the Neumaier bound is called a \emph{Neumaier clique}, and is $\av \a$-regular, where $\av \a = s-\av \theta$ (compare \thref{rmk:alpha}). An edge-regular graph with a Neumaier clique is called a \emph{Neumaier graph}. If a Neumaier graph is not strongly regular, then it is called \emph{strictly Neumaier}. Research is currently being done on constructing and classifying strictly Neumaier graphs, see for example \cite{StrictlyNeumaier1, StrictlyNeumaier2, Neumaier, StrictlyNeumaier3, StrictlyNeumaier4, StrictlyNeumaier5}. Similarly as how Hoffman's bound on the chromatic number follows from (\ref{eq:classic}) and the ratio bound (\ref{eq:ratiobound}), we obtain a bound on the chromatic number from the Neumaier bound.

\begin{thm}\thlabel{thm:co-edge}
    If $G$ is co-edge-regular, then $\chi(G) \ge \av s +1$.
\end{thm}

\begin{proof}
    Since $\ol G$ is edge-regular, we can apply the Neumaier bound (\thref{prop:Neumaier}) to get $\a(G) \le \ol s+1$. Together with (\ref{eq:classic}), \thref{lem:(s+1)(ols+1)=n} now gives the result.
\end{proof}

Note that by (\ref{eq:h<s+1}), the bound given in \thref{thm:co-edge} is a strict improvement of Hoffman's bound for co-edge-regular graphs that are not strongly regular.

\subsubsection*{Applications to Hoffman colorability}

We can apply the results from this section to Hoffman colorings of regular graphs. We derive three results (\thref{cor:sufficient}, \thref{cor:w<s+1}, \thref{cor:numtriangles}), each of which can be seen as a necessary condition for Hoffman colorability of regular graphs, or as a characterization of strongly regular graphs among Hoffman colorable regular graphs.

Firstly, we present four sufficient conditions for a Hoffman colorable regular graph to be strongly regular.

\begin{cor}\thlabel{cor:sufficient}
    Let $G$ be a Hoffman colorable regular graph. Then $G$ is strongly regular if any of the following holds:
    \begin{enumerate}[label=(\roman*)]
        \item $G$ is co-edge-regular,
        \item $G$ is a Neumaier graph,
        \item $G$ has a Hoffman clique,
        \item $\ol G$ is Hoffman colorable.
    \end{enumerate}
\end{cor}
\begin{proof}
   Since $G$ is Hoffman colorable, by (\ref{eq:h<s+1}) it suffices to show that $\chi(G)\ge \av s+1$. Now Part (i) follows from \thref{thm:co-edge}. A Neumaier clique is of size $\av s+1$, so for Neumaier graphs $\chi(G)\ge \av s+1$, implying Part (ii). Part (iii) follows similarly from (\ref{eq:s ratio}). Part (iv) follows from \thref{prop:constantequitable} and Part (iii).
\end{proof}

The first two conditions can be rephrased as follows: no strictly Neumaier graph is Hoffman colorable, and no co-edge-regular graph that is not strongly regular is Hoffman colorable.

\begin{exa}
Examples of graphs satisfying \thref{cor:sufficient}(iv) are the square lattice graphs $L(K_{m,m})$ and their complements (the columns form a spread in $L(K_{m,m})$, and the translates of the diagonal form a spread in the complement), and Paley graphs of a square order (the cosets of $\F_q$ in $\F_{q^2}$ give a spread in the Paley graph over $\F_{q^2}$, and since Paley graphs are self-complementary this gives also a Hoffman coloring), providing an example of a regular self-complementary Hoffman colorable graph.
\end{exa}

Secondly, we can derive that the Neumaier bound holds not only for edge-regular graphs, but also for Hoffman colorable regular graphs.

\begin{cor}\thlabel{cor:w<s+1}
    If $G$ is a Hoffman colorable regular graph, then $$\omega(G) \le  \av s+1,$$ with equality if and only if $G$ is strongly regular and has a Delsarte clique.
\end{cor}

\begin{proof}
    We have $\omega(G)\le \chi(G) = h(G) \le \av s+1$, where the last part follows from (\ref{eq:h<s+1}). In case of equality, we have equality in (\ref{eq:h<s+1}) so $G$ is strongly regular, and we have $\omega(G)=h(G)$ so $G$ has a Delsarte clique. The other direction is clear.
\end{proof}

Note that even though equality in the Neumaier bound is not understood completely yet, \thref{cor:w<s+1} gives a comprehensive condition for equality. Note that we have equality in \thref{cor:w<s+1} if and only if $G$ is strongly regular and not a core, by \thref{thm:SRGcores}.

As a third and final application, we obtain the following bound on the number of triangles of a Hoffman colorable regular graph.

\begin{cor}\thlabel{cor:numtriangles}
    If $G$ is a Hoffman colorable regular graph with $N$ triangles, then
    $$N \ge  \frac{kn}{6} \Big (\chi -1 + \frac{(k-\chi) (k-\chi+1)}{n-\chi}-\frac{k(n-k-1)}{(n-\chi)(\chi-1)} \Big ),$$
    with equality if and only if $G$ is strongly regular.
\end{cor}
\begin{proof}
    By double counting the pairs of incident edges and triangles, we get $3N=kn \av a/2$. By (\ref{eq:h<s+1}), we know for a Hoffman colorable regular graph that $\chi \le \av s +1$ with equality if and only if the graph is strongly regular. This means that $\chi-1 \le \av s$, and evaluating the polynomial from (\ref{eq:s}) at $\chi -1$ results in a non-positive number. The result now follows after isolating $\av a$ in this inequality.
\end{proof}

\subsection*{Acknowledgments}

Aida Abiad is supported by NWO (Dutch Research Council) through the grants VI.Vidi.213.085 and OCENW.KLEIN.475. The research of Thijs van Veluw is supported by the Special Research Fund of Ghent University through the grant BOF/24J/2023/047. The authors thank Willem Haemers for inspiring discussions on this topic, and David Roberson for pointing out the relation between \thref{thm:characterization} and one of his inequalities.

\end{document}